\documentclass[]{interact}
\usepackage[linktoc=page]{hyperref}%page
\hypersetup{linkcolor  = black}
\hypersetup{citecolor=black} %blue
\hypersetup{colorlinks=true}
\hypersetup{urlcolor=cyan}
\usepackage{color, colortbl}
\usepackage{graphicx, amsmath, amsthm, amssymb, mathrsfs, amscd}
\usepackage{epstopdf}
\usepackage[caption=false]{subfig}
\usepackage[numbers,sort&compress]{natbib}% Citation support using natbib.sty
\bibpunct[, ]{[}{]}{,}{n}{,}{,}% Citation support using natbib.sty
\makeatletter% @ becomes a letter
\def\NAT@def@citea{\def\@citea{\NAT@separator}}% Suppress spaces between citations using natbib.sty
\makeatother% @ becomes a symbol again
\theoremstyle{plain}% Theorem-like structures provided by amsthm.sty
\newtheorem{theorem}{Theorem}[section]
\newtheorem{lemma}[theorem]{Lemma}

\theoremstyle{definition}
\newtheorem{definition}[theorem]{Definition}

\theoremstyle{remark}
\newtheorem{remark}{Remark}

\begin{document}
%\articletype{ARTICLE TEMPLATE}% Specify the article type or omit as appropriate
\title{Zeros of a Two-parameter Family of Harmonic Quadrinomials}
\author{
\bigskip \name{Oluma.A~Alemu\textsuperscript{1} and Hunduma~L.~Geleta\textsuperscript{2}}
 \affil{\textsuperscript{1} Department of Mathematics, Addis Ababa University, Addis Ababa, Ethiopia. \\  Email: oluma.ararso@aau.edu.et \\ \bigskip \textsuperscript{2} Department of Mathematics, Addis Ababa University, Addis Ababa, Ethiopia. \\  Email:hunduma.legesse@aau.edu.et}}
\maketitle
\begin{abstract}
In this paper, we determine the number of zeros and the zero inclusion regions of a two-parameter family of harmonic quadrinomials. We also determine a curve that separates sense-preserving and sense-reversing regions for these families of quadrinomials. Our work makes practical and effective use of the work of Wilmshurst, Khavinson, Dehmer, and also Bezout’s theorem in the plane.
\end{abstract}
\section{Introduction}$\label{I}$
%We study here complex-valued harmonic functions in the plane.  The behavior of such a function can be vastly different from that of an analytic function. Analytic functions
%are preserved under composition, but harmonic functions are not. A harmonic function of an analytic function is harmonic, but an analytic function of a harmonic function need not be harmonic. The analytic functions form an algebra, but the harmonic functions do not. Even the square or the reciprocal of a harmonic function need not be harmonic. The inverse of a harmonic mapping need not be harmonic. The boundary behavior of harmonic mappings may be much more complicated than that of conformal mappings. It will be seen, nevertheless, that much of the classical theory of conformal mappings
%can be carried over in some way to harmonic mappings. Analytic polynomials take every value a finite number of times. In contrast, the range of a harmonic polynomial can  exclude an open region of the complex plane. Analytic functions  are orientation-preserving (or  sense-preserving) at all points at which the derivative is not zero. However, complex-valued harmonic functions may be sense-preserving in one region and sense-reversing in  \bigskip another. For further properties, see \cite{duren2004harmonic}.  \\ 
A complex -valued harmonic polynomial on a simply connected domain $\mathbb{D}$ is a function that can be decomposed as $f(z) = h(z) + \overline{g(z)},$ where $h$ and $g$ are polynomials of degree $n$ and $m$ respectively(see \cite{brilleslyper2012explorations} , chapter 4). In this decomposition,  $h(z)$ is called analytic \bigskip part and $g(z)$ is said to be co-analytic part of $f.$  \\ 
In 1940 Kennedy  \cite{kennedy1940bounds} bounds the roots of analytic trinoimial equations of the form $z^n + az^k + b = 0 ,$ where $ab \neq 0$ by asserting that the roots of this type of trinomials have certain bounds to their respective absolute values. In 2012,  Melman \cite{melman2012geometry} determined the regions in which the zeros of analytic trinomials of the form $p(z) = z^n - \alpha z^k - 1,$ with integers $n \geq 3$ and $1 \leq k \leq n - 1$ with $\mathrm{gcd}(k, n) = 1,$ and $\alpha \in \mathbb{C},$ lie. He determines the zero inclusion regions for the following cases: \textit{\textbf{a$)$}} for any value of $|\alpha|;$ \textit{\textbf{b$)$}}  $| \alpha| > \sigma (n, k);$ and \textit{\textbf{c$)$}}  $| \alpha| < \sigma (n, k)$ where a separability threshold, $\sigma(n, k)$, is given by  $\sigma(n, k)=\frac{n}{n-k} (\frac{n-k}{k})^{\frac{k}{n}}.$ In all these  three cases he told  useful  information  on the \bigskip location  of the zeros of $p.$ \\
Determining the number of zeros of complex-valued harmonic polynomial and locating the zeros are attracting problems in complex analysis. Let $\mathcal{Z}_f$ denote the number of zeros of $f,$ that is, number of points $z \in \mathbb{C}$ satisfying $f(z) = 0.$ For  $f(z) = h(z) + \overline{g(z)}$ with $\mathrm{deg}h=n>m=\mathrm{deg}g,$  we have $n \leq \mathcal{Z}_f \leq n^2.$ The lower bound is sharp for each $m$ and $n$ due to the generalized argument principle. Note that the zeros of $f$ are isolated. Therefore by applying  Bezout’s theorem to the real and imaginary parts of $f(z) = 0,$ the upper bound follows. This was shown by Wilmshurst \cite{wilmshurst1998valence}. For instance as illustrated by Wilmshurst in the same paper,  $\left(\frac{1}{i}\right)^n Q\left(iz+\frac{1}{2} \right) $ is a polynomial with $n^2$  zeros where $Q(z)=z^n+(z-1)^n+i\overline{z}^n -i(\overline{z}-1)^n.$ For analytic polynomial $f$ of degree $n,$ Fundamental theorem of Algebra stipulates that $f$ has $n$ total number of zeros counting \bigskip with multiplicity.\\
 Recently, Brilleslyper \textit{et al} \cite{brilleslyper2020zeros} studied on the number of zeros of harmonic trinomials of the form $p_c(z) = z^n + c\overline{z}^k - 1$ where $1 \leq k \leq n - 1, n \geq 3, c \in \mathbb{R^+},$ and $\mathrm{gcd}(n, k) = 1.$ They showed that the number of zeros of  $p_c(z)$ changes as $c$ varies and proved that the distinct number of zeros of $p_c(z)$ ranges from $n$ to $n+2k.$  Among other things, they used the argument principle for harmonic function that can be formulated as a direct generalization of the classical result for analytic functions. 
 %In \cite{hundolum2021harmoniczeros}, we also showed how the location of the zeros of this trinomials depends on the real parameter $c.$
  Now we move from a one parameter trinomial to a two parameter quadrinomials and we are interested in how  the number of the zeros and location of the zeros changes as a real  parameters $b,c \in \mathbb{R}$ \bigskip varies. \\ 
In this paper, we first determine the zero inclusion regions of the zeros of quadrinomials of the form $q(z) = bz^k+\overline{z}^n +c\overline{z}^m +z$ where $k,m, n \in \mathbb{N}$ with $n>m,$ and $ b,c \in \mathbb{R}$ by considering different cases. Then we determine the possible maximum number  of the zeros of these families of harmonic quadrinomials. Finally, we determine the a curve that separates sense-preserving and sense-reversing region for this quadrinomial. The main theorems in this paper are Theorem  $\ref{III5}$ , Theorem $\ref{III6},$  Theorem $\ref{III1},$ and \bigskip Theorem $\ref{III7}.$\\
 This paper is organized as follows. In section $\ref{II},$ we present some important preliminary results that will be used to proof the main theorems. In section $\ref{III},$ We state and proof  main theorems. Theorems $\ref{III5}$ and $\ref{III6}$ determines the zero inclusion regions for the families of quadrinomials $q(z) = bz^k+\overline{z}^n +c\overline{z}^m +z$ where $k,m, n \in \mathbb{N}$ with $n>m,$ and $ b,c \in \mathbb{R}$ by taking different cases and we are interested in how the number and the location of zeros of $q(z)$ changes as $b$ and  $c$ varies. Theorem $ \ref{III1}$ determines the maximum number of  zeros for our quadrinomials.  Also we look for some properties of the zeros of $q(z)$ in Theorem $\ref{III7}.$ 
\section{Preliminaries} $\label{II}$ 
In this section we review some important concepts and results that we will use later on to prove main theorems. We begin by stating the well known results and some useful definitions, theorems and lemmas.
\begin{theorem}[\textbf{Lewy's Theorem},\cite{lewy1936non}]\label{2.1} 
If $f$ is a complex valued harmonic function that is locally univalent in a domain $D \subset \mathbb{C},$ then its Jacobian $,J_f(z),$ never vanish for all $z \in D.$ 
\end{theorem}
As an immediate consequence of Theorem $\ref{2.1}$, a complex valued harmonic function $f(z)=h(z)+\overline{g(z)}$ is locally univalent and sense-preserving if and only if $h'(z) \neq 0$ and $|\omega (z)|<1,$ where $\omega (z)$ is a dilatation  function of $f$ defined by $\omega (z) = \frac{g'(z)}{h'(z)}.$  A harmonic function $f(z) = h(z) + \overline{g(z)},$ is called sense-preserving at $z_0$ if the Jacobian $J_f(z) > 0$ for every $z$ in some punctured neighborhood of $z_0.$ We also say that $f$ is sense-reversing if $\overline{f}$ is sense-preserving at $z_0.$ If  $f$ is  neither sense-preserving nor sense-reversing at $z_0,$ then $f$  is said to be a singular \bigskip polynomial at $z_0.$\\ Recall to Argument principle for harmonic functions in \cite{duren1996argument}: Let $f$ be a complex valued harmonic function in a Jordan domain $D$ with boundary $\Gamma .$ Suppose that $f$ is continuous in $\overline{D}$ and $f \neq 0$ on $\Gamma .$ Suppose that $f$ has no singular zeros in $D,$ and let $N = N_+ - N_- ,$ where $N_+$ and $N_-$ are the number zeros in sense-preserving region and sense-reversing \bigskip region of $f$ in $D,$ respectively. Then, $\bigtriangleup _\Gamma arg f(z)  = 2 \pi N.$
\begin{definition}
A point $\zeta$ is called a critical point of a polynomial $f$ if $f'(\zeta) = 0.$  A fixed point $\zeta \in \mathbb{C}$ is \textbf{attractive, repelling} or \textbf{neutral} if, respectively, $|f'(\zeta)| < 1, |f'(\zeta)| > 1$ or
$|f'(\zeta)| = 1.$  A neutral fixed point is rationally neutral if $f'(\zeta)$ is a root of unity. We shall say that a fixed point $\zeta$ attracts some point $w \in \mathbb{C}$ provided that the sequence $f^k(w) = \underbrace{f(w) \circ f(w) \circ \cdots \circ f(w)}_{k-copies}$ converges to $\zeta.$ 
\end{definition}
Note that as mentioned in \cite{carleson2013complex}, if $deg f > 1$ and  $\zeta$ is an attracting or rationally neutral fixed point, then $\zeta$ attracts some critical point of $f.$
%\begin{definition}[\cite{neumann2005valence}] 
%The \textbf{\textit{valence of a function}} $f$ at a given point $w$, denoted $Val(f, w),$ is the number of distinct points $z$ in the domain of $f$ such that $f(z) = w.$ The valence of a function, denoted $Val(f),$ is the supremum of $Val(f, w).$
%\end{definition}
\begin{definition}[\cite{khavinson2003number}]
Let $f(z)$ be a polynomial   such that the conditions $ |f'(z_0)| = 1$ and $\overline{f(z_0)} = z_0$ are not satisfied simultaneously for any $z_0 \in \mathbb{C}.$ Then $f(z)$ is said to be a \textbf{\textit{regular}} polynomial.
\end{definition}
\begin{definition}[\cite{brilleslyper2014locating}]
The roots that lie on the unit circle are referred to as  \textit{uni-modular roots.} 
\end{definition}
\begin{definition}
The grand orbit  of a point $z \in \hat{\mathbb{C}},$ denoted by $[z],$ under a map $f$ is defined as the set of all points $w$ such that there exists $m, n \geq ~0$ with $f^m(z) = f^n(w).$
\end{definition}
\textbf{Fact 1:} A grand orbit under a transformation $T$ is an equivalence class of the relation $x \sim y~~ iff ~~T^p(x) = T^q(y)$ for some $p, q > 0.$
\begin{theorem}[\textbf{Descartes' Rule of Signs} \cite{wang2004simple}] $\label{p''}$
    Let $p(x) = a_0x^{b_0}+ a_1x^{b_1}+ \cdots + a_nx^{b_n}$ denote a polynomial with nonzero real coefﬁcients $a_i,$ where the $b_i$ are integers satisfying $0 \leq b_0 < b_1< b_2< \cdots < b_n.$ Then the number of positive real zeros of $p(x)$ (counted with multiplicities) is either equal to the number of variations in sign in the sequence $a_0, \cdots , a_n$ of the coefficients or less than that by an even whole number. The number of negative zeros of p(x) (counted with multiplicities) is either equal to the number of variations in sign in the sequence of the coefficients of $p(-x)$ or less than that by an even whole number.
   \end{theorem}
   Note that according to Descartes's Rule of Signs, polynomial $p(x)$ has no more positive roots than it has sign changes and  has no more negative roots than  $p(-x)$  has sign changes. A polynomial may not achieve the maximum allowable number of roots given by the Fundamental Theorem, and likewise it may not achieve the maximum allowable number of positive roots given by the Rule of Signs.
\begin{theorem}[Bezout's Theorem In the Plane, \cite{kirwan1992complex}]$\label{II0}$
  Let $f$ and $g$ be relatively prime polynomials in the real variables $x$ and $y$ with real coefficients, and let $\mathrm{deg}h = n$ and $\mathrm{deg}g = m.$ Then the two algebraic curves $f(x, y) = 0$ and $g(x, y) = 0$ have at most $mn$ points in common.
  \end{theorem}
  Bezout's theorem is one of the most fundamental results about the degrees of polynomial surfaces and it bounds the size of the intersection of polynomial surfaces.
  \begin{theorem}[\cite{khavinson2018zeros}]
  For a harmonic polynomial $f(z) = h(z) + \overline{g(z)}$ with real coefficients, the equation $f(z) = 0$ has at most $ n^2 - n$ solutions that satisfy $(Re z)(Imz) \neq 0$ where $\mathrm{deg}h =n>m=\mathrm{deg}g.$
  \end{theorem}
\begin{theorem}[ Wilmshurst , \cite{wilmshurst1998valence}]$\label{II1}$
If $f(z)= h(z)+\overline{g(z)} $ is a harmonic polynomial such that $\mathrm{deg}h =n>m=\mathrm{deg}g$ and $\lim_{z \rightarrow \infty}f(z) = \infty,$ then $f(z)$ has at most $n^2$ zeros.
\end{theorem}
The proof of Theorem $\ref{II1} $ can also readily follows from Bezout’s theorem. Let $f(z)=h(z) +\overline{g(z)}= 0$  with $h=u+iv$ and $g=\alpha +i\beta.$ Then $f(z)= \left( u(x,y)+\alpha(x,y)\right)+i\left(v(x,y)-\beta(x,y)\right)=0+i0.$ Now define $\eta (x,y)
:= u(x,y)+\alpha(x,y)$ and $\zeta (x,y):=v(x,y)-\beta(x,y).$ Then we do have a homogeneous polynomial equations equation 
\begin{equation}
\begin{cases}
\eta (x,y)=0\\
 \zeta (x,y)=0.
 \end{cases}
 \end{equation}
 Here we know that $\mathrm{deg} \eta =n$ and $\mathrm{deg}\zeta=n.$ Therefore by Bezout's Theorem the maximum number of zeros of $\eta$ and $\zeta$ in common is $n^2.$ 
\begin{theorem}$\label{II2}$
Let $q(z)= z-\overline{p(z)},$ where $p(z)$ is an analytic polynomial with $\mathrm{deg}p=n>1.$  Then $\mathcal{Z}_q \leq 3n-2.$
\end{theorem}
For the proof, see \cite{khavinson2003number}, Theorem 1.
\begin{theorem}[\textbf{Dehmer, M.} \cite{dehmer2006location}]$\label{II3}$
   Let $f(z) = a_nz^n+a_{n-1}z^{n-1}+ \cdots +a_1z+a_0, ~~~ a_n \neq 0$ be a complex polynomial. All zeros of $f(z)$ lie in the closed disc $K(0,\mathrm{max}(1,\delta)),$ where $M:= \mathrm{max}\{ |\frac{a_j}{a_n}| \} ~~\forall j=0,1,2,\cdots ,n-1$ and $\delta \neq 1$ denotes the positive root of the equation $z^{n+1} - (1+M)z^n +M =0.$ 
   \end{theorem}
   Note that this theorem is also another classical result for the location of the zeros of analytic complex polynomial which depends on algebraic equation's positive root. Descartes' Rule of Signs plays a great role in the proof as shown by Dehmer. 
   \begin{lemma}$\label{II4}$
  A complex valued harmonic function $f(z)=h(z)+\overline{g(z)}$ is analytic if and only if $f_{\overline{z}} =0.$
   \end{lemma}
\section{Main results} $\label{III}$ 
Under this section, we find the maximum number of zeros and bound all the zeros of complex valued harmonic quadrinomials in a closed disk and come up with  a certain conclusion.  
\subsection{On the location of the zeros of $q(z)$}
In this section we find an upper bound on the moduli of all the zeros of complex valued harmonic quadrinomial  $q(z) = bz^k+\overline{z}^n +c\overline{z}^m +z$ , where $k,m, n \in \mathbb{N}$ with $n>m,$ and $ b,c \in \mathbb{R}.$ The following theorem derives a closed disk in which all zeros are included and hence it bounds the moduli of the zeros of this families of the complex valued harmonic quadrinomials.
\begin{theorem}$\label{III5}$
Let $q(z) = bz^k+\overline{z}^n +c\overline{z}^m +z$ where $k,m, n \in \mathbb{N}$ with $n>m$ and $b,c \in \mathbb{R}\setminus \{0\}.$  Let $|c|>1$ and $k>n.$ If $\delta _1 \neq 1$ is the positive real root of the equation
 \begin{equation}
|b|x^{k+1}-(|b|+|c|)x^k+|c| =0,
\end{equation} then all the zeros of $q(z)$ lie in the closed disk $D(0;R)$ where $R= \mathrm{max}\{ 1, \delta _1 \}.$
\end{theorem}
\begin{proof}
Consider $q(z) = bz^k+\overline{z}^n +c\overline{z}^m +z$ with given conditions. This quadrinomial can be rewritten as 
\begin{equation}
q(z)= \left( bz^k +z \right) + \overline{ \left( z^n+cz^m \right)}.
\end{equation}  Since both $bz^k +z$ and $z^n+cz^m$ are analytic complex-valued polynomials, $q(z)$ is a harmonic polynomial. Note that $b$ is a parameter in analytic part and $c$ is a parameter in co-analytic part of this quadrinomial.  Therefore using the triangle inequality, we have 

$$|q(z)| = |\left( bz^k +z \right) + \overline{ \left( z^n+cz^m \right)}| \geq \left| |bz^k|- \left[ \left| z+ \overline{ \left( z^n+cz^m \right)} \right| \right] \right|.$$ Since $$\left| z+ \overline{ \left( z^n+cz^m \right)} \right| \leq |z| + \left| \overline{ \left( z^n+cz^m \right)} \right|$$ and the modulus of the conjugate of any complex number equals with its modulus, the direct forward calculation gives us: $$|q(z)| \geq |b|\left[ |z|^k - \left(\frac{1}{|b|}|z| + \frac{1}{|b|}|z|^n + \frac{|c|}{|b|}|z|^m \right) \right] .$$ Then we have $$|q(z)| \geq |b|\left[|z|^k-\frac{|c|}{|b|} \left( |z|+|z|^n+|z|^m \right)\right].$$ Since also we have assumed that $k>n>m,$ we have  $$ |q(z)| \geq |b|\left[ |z|^k- \frac{|c|}{|b|}\left( |z|^k +|z|^{k-1}+ \cdots +1\right) \right]$$ $$ = |b|\left[|z^k| - \frac{\frac{|c|}{|b|}\left( |z|^k-1 \right)}{|z|-1} \right]~~~~~~~~$$ $$~~~~=|b|\left[ \frac{|z|^{k+1} - \left( 1+\frac{|c|}{|b|}\right)|z|^k +\frac{|c|}{|b|}}{|z|-1}\right]  .$$ Descartes’ Rule of Signs gives us that the function $x^{n +1} - \left(1 + \frac{|c|}{|b|}\right)x^n + \frac{|c|}{|b|}$ has exactly two distinct positive zeros, since $x = 1$ is a root with multiplicity one \cite{wang2004simple}. We then conclude that $|q(z)| \neq 0~~ \forall z \in \mathbb{C} \setminus D(0, R).$ Hence all the zeros of $q(z)$ lie in the closed disk $D(0, R).$
\end{proof}
\begin{theorem}$\label{III6}$
Let $q(z) = bz^k+\overline{z}^n +c\overline{z}^m +z$ where $k,m, n \in \mathbb{N}$ with $n>m$ and $b,c \in \mathbb{R}\setminus \{0\}.$ Let $|c|\leq 1$ and $k>n.$ If $\delta _2 \neq 1$ is the positive real root of the equation
 \begin{equation}
|b|x^{k+1}-(|b|+1)x^k+1 =0,
\end{equation} then all the zeros of $q(z)$ lie in the closed disk $D(0;R)$ where $R= \mathrm{max}\{ 1, \delta _2 \}.$
\end{theorem}
\begin{proof}
If  $|c|\leq 1,$ then $\max(\frac{1}{|b|}, \frac{|c|}{|b|})= \frac{1}{|b|}.$ Since $\delta _2 \neq 1$ is a positive real root of  $|b|x^{k+1}-(|b|+1)x^k+1 =0,$ it is also a positive real root of $x^{k+1}-(1+ \frac{1}{|b|})x^k+\frac{1}{|b|} =0.$ Thus we have for a complex valued harmonic polynomial $q(z) = bz^k+\overline{z}^n +c\overline{z}^m +z$ with $k>n,$ $\delta _2 \neq 1$ is a positive real root of  $x^{k+1}-(1+ \frac{1}{|b|})x^k+\frac{1}{|b|} =0.$ Then we have $$|q(z)|= |bz^k+\overline{z}^n +c \overline{z}^m +z|$$ $$~~~~~~~~~~ \geq \left||bz^k|-|z+\overline{z}^n +c \overline{z}^m|\right| $$ By triangle inequality, we then have $$ |q(z)| \geq |b|\left[ |z|^k -\left( \frac{1}{|b|}|z| +\frac{1}{|b|}|z|^n + \frac{|c|}{|b|}|z|^m \right) \right]$$ $$ \geq |b| \left[ |z|^k -\frac{1}{|b|}\left( |z|+|z|^n+|z|^m \right) \right] $$ Since also we have assumed that $k>n>m,$ we have  $$ |q(z)| \geq |b|\left[ |z|^k- \frac{1}{|b|}\left( |z|^k +|z|^{k-1}+ \cdots +1\right) \right]$$ $$ = |b|\left[|z^k| - \frac{\frac{1}{|b|}\left( |z|^k-1 \right)}{|z|-1} \right]~~~~~~~~$$ $$~~~~=|b|\left[ \frac{|z|^{k+1} - \left( 1+\frac{1}{|b|}\right)|z|^k +\frac{1}{|b|}}{|z|-1}\right]  .$$ Descartes’ Rule of Signs gives us that the function $x^{n +1} - (1 + \frac{1}{|b|})x^n + \frac{1}{|b|}$ has exactly two distinct positive zeros, since $x = 1$ is a root with multiplicity one \cite{wang2004simple}. We then conclude that $|q(z)| \neq0 ~\forall z \in \mathbb{C} \setminus D(0, R).$ Hence all the zeros of $q(z)$ lie in the closed disk $D(0, R).$
\end{proof}
\begin{remark}
For a non-zero real numbers $b$ and $c,$ the location of the zeros of quadrinomial $q(z) = bz^k+\overline{z}^n +c\overline{z}^m +z$ varies as the coefficient of co-analytic part varies.
\end{remark}
\subsection{On the number of the zeros of $q(z)$}
In this section we find an upper bound for the possible number of the  zeros of  quadrinomial \bigskip $q(z) = bz^k+\overline{z}^n +c\overline{z}^m +z$ , where $k,m, n \in \mathbb{N}$ with $n>m,$ and $ b,c \in \mathbb{R}.$ \\
Now consider the harmonic quadrinomial $q(z) = bz^k+\overline{z}^n +c\overline{z}^m +z.$ This function can be rewritten as
  \begin{equation}
   q(z)= z-\overline{(-b\overline{z}^k- z^n - cz^m)}.
  \end{equation}
Then Theorem  $\ref{II1}$  and  Theorem $\ref{II2}$ plays a great role to arrive at the following theorem. We prove this theorem by breaking it into two cases and it bounds the number of zeros of quadrinomial from above.
\begin{theorem}$\label{III1}$
Let $q(z) = bz^k+\overline{z}^n +c\overline{z}^m +z$ where $k,m, n \in \mathbb{N}$ with $k>n>m,$  $ b,c \in \mathbb{R}$ and $\mathcal{Z}_q$ denotes the number of the zeros of $q(z).$ Then
 \begin{equation}
 \mathcal{Z}_q \leq  \begin{cases} 3n-2,~~  if ~~ b=0 \\ n(n-1)+3k-2,~~if~~b \neq 0~~ and ~~n<k-1 \\ k^2,~~ if~~ b \neq 0 ~~and~~n=k-1 \end{cases}.
 \end{equation}  
\end{theorem}
To proof this theorem let us consider the following lemmas:
\begin{center}
\textbf{\textsl{\underline{The case for $b=0$}}}
\end{center}
\begin{lemma}$\label{III2}$
A complex valued harmonic function $b\overline{z}^k+z^n +cz^m$ is analytic if and only if  $b=0.$
\end{lemma}
\begin{proof}
By lemma $\ref{II4}$ , the function $b\overline{z}^k+z^n +cz^m$ is analytic if and only if $$\frac{\partial}{\partial \overline{z}}(b\overline{z}^k+z^n +cz^m) =0.$$ But this is true if and only if $bk\overline{z}^{k-1}=0$ which directly implies that $b=0.$  
\end{proof}
Note that if $b=0,$ then we have $q(z)= z+\overline{z^n+cz^m}= z-\overline{p(z)}$ where $p(z)=-z^n-cz^m.$ Since $deg p=n,$ we have the following lemmas and we can prove them directly by the same procedure as in \cite{khavinson2003number}.  Also under this case, note that a harmonic function  $\overline{z}^n +c\overline{z}^m +z$ where $m, n \in \mathbb{N}$ with $n>m,$ and $c \in \mathbb{R}$  at a point $z_0,$ is sense-preserving if and only if $\mathrm{2Re}z^{n-m} < \frac{|z|^{2(1-m)}}{cmn} - \frac{n|z|^{2(n-1)}}{cm}-\frac{cm}{n} ,$ sense-reversing if and only if  $\mathrm{2Re}z^{n-m} > \frac{|z|^{2(1-m)}}{cmn} - \frac{n|z|^{2(n-1)}}{cm}-\frac{cm}{n} ,$  and  singular if and only if $\mathrm{2Re}z^{n-m} = \frac{|z|^{2(1-m)}}{cmn} - \frac{n|z|^{2(n-1)}}{cm}-\frac{cm}{n}.$
\begin{lemma}$\label{III3}$
The cardinality of the set of points  $$ \{ z \in \mathbb{C}: z= \overline{-z^n -cz^m}~~ and~~ |nz^{n-1}+cmz^{m-1}| \leq 1 \}$$ is  at most $n-1$ where $m, n \in \mathbb{N}$ with $n>m,$ and $ c \in \mathbb{R}.$  
\end{lemma}
\begin{proof}
Let $G(z):= \overline{-\overline{z}^n-c\overline{z}^m};$ where $g(z)=-z^n-cz^m.$ Here $G(z)$ is analytic polynomial of degree $n^2.$ If $|g'(z_0)|=1$ and $g(z_0)=z_0,$ then $G'(z_o)=1.$ All points in the set $\{ z \in \mathbb{C}: z= -\overline{z}^n-c\overline{z}^m ~~ and ~~ |nz^{n-1}+cmz^{m-1}| \leq 1 \}$ are fixed points of $G(z),$ which are either attracting or rationally neutral.So each of them attracts a critical point of $G.$ As it was proved by Khavinson \cite{khavinson2003number}, if $-\overline{z_0}^n-c\overline{z_0}^m$ and $c \in \mathbb{C},$ then $\left( -\overline{z}^n-c\overline{z}^m \right)^k \rightarrow z_0$ if and only if $G^k(z) \rightarrow z_0.$ 
If $G'(c)=0,$ then there are atleast $n+1$ critical points of $G,$ counted with multiplicities, which all belongs to the same grand orbit under $\overline{g(z)}.$  Furthermore, if $G'(c)=0, g(z_0)=\overline{z_0}$ and $G^k(c) \rightarrow z_0,$ then there are $n+1$ critical points of $G,$ counted with multiplicities, all attracted to $z_0$ under the iteration of $G.$ Each point in the set $$\{ z\in \mathbb{C}: z= -\overline{z}^n-c\overline{z}^m ~~ and ~~ |nz^{n-1}+cmz^{m-1}| \leq 1 \}$$  attracts a critical point of $G,$ but then it attracts $n+1$ of them. Since different fixed points attract a disjoint set of critical points and the degree of $G(z)$ is $n^2,$ the total number of its critical points counted with multiplicities is $n^2-1.$
\end{proof}
\begin{lemma}$\label{III4}$
If $\overline{z}^n +c\overline{z}^m $ is regular polynomial, then  the cardinality of the set $$ \{ z \in \mathbb{C}:z= \overline{-z^n -cz^m}~~ and~~ |nz^{n-1}+cmz^{m-1}| > 1 \} $$ is at most $2n-1$ where $m, n \in \mathbb{N}$ with $n>m,$ and $ c \in \mathbb{R}.$
\end{lemma}
\begin{proof}
Rewrite $q(z):= z-\overline{g(z)},$ where $g(z)= -z^n-cz^m.$ Let $\Gamma _+$ be a region where $z-\overline{g(z)}$ is sense-preserving and $\Gamma _ -$ be a region where $z-\overline{g(z)}$ is sense-reversing. Then $\Gamma_+$ and $\Gamma_-$ are separated by $\partial \Gamma_+.$ Moreover, make $\Gamma_-^0$ compact by intersecting $\Gamma_-$ with a large disk $D(0,R)$ chosen so that $\partial \Gamma_+ \subset D(0,R),$ all zeros of $z-\overline{g(z)}$ are in $D(0,R)$ and the argument change of $z-\overline{g(z)}$ along the circle $C(0,R)$ is $-n.$ Then by argument principle and lemma $\ref{III3}$, $$ \bigtriangleup _{\partial \Gamma _+} \left( z-\overline{g(z)} \right) \leq 2 \pi (n-1).$$ Hence, $$ \bigtriangleup_{C(0,R)} -\bigtriangleup_{\partial \Gamma _+} \geq -2\pi(2n-1).$$ Since $C(0,R)-\partial \Gamma _+$ is the oriented boundary of $\Gamma _-^0,$ argument principle means that $\frac{-1}{2 \pi}\left[  \bigtriangleup_{C(0,R)} -\bigtriangleup_{\partial \Gamma _+} \right]$  is the number of zeros of $z-\overline{g(z)}$ in $\Gamma_-.$ This proves our lemma.
\end{proof}
Now by using lemma $\ref{III3}$ and lemma $\ref{III4}$,  we have 
\begin{equation}
\mathcal{Z}_q \leq (n-1) +(2n-1)=3n-2
\end{equation} for $b=0.$
\begin{center}
                        \textbf{\textsl{\underline{The case for $b \neq ~0$}}}
\end{center}
In his thesis Wilmshurst \cite{wilmhurst1994complex} has constructed a number of examples of harmonic polynomials $f(z)=h(z)+\overline{g(z)}$of degree $n$ with 'maximal valence' $n^2.$ If we consider our quadrinomial $q(z) = bz^k+\overline{z}^n +c\overline{z}^m +z,$ 
\begin{equation}
\lim_{z \rightarrow \infty}(bz^k+\overline{z}^n +c\overline{z}^m +z) = \infty.
\end{equation} Therefore it follows from Theorem $\ref{II1}$ that the number of zeros of $q(z)$ is bounded above by $k^2.$ That means, $\mathcal{Z}_q \leq k^2.$  In all such examples he found that, if $h(z)$ had degree $n$, then $g$ had degree at least $n - 1.$ On the basis of this and other considerations Wilmshurst has made the following conjecture: if $g$ has degree $m$ and $h$ has degree $n > m,$ then $f$ has at most $m(m - 1) + 3n -2$ distinct zeros. Hence by this conjucture, it follows that $\mathcal{Z}_q \leq n(n-1)+3k-2.$ This finishes the proof of Theorem $\ref{III1}.$
\subsection{Further analysis for $n=k>m=1$}
Under this section, we determine a curve that separates sense-preserving and sense-reversing region for the family of quadrinomials of the form $q(z) = bz^k+\overline{z}^n +c\overline{z}^m +z$ , \bigskip where $k,m, n \in \mathbb{N}$ with $n=k>m=1,$ and $ b,c \in \mathbb{R}.$\\
Note that a quadrinomial $q(z) = bz^k+\overline{z}^n +c\overline{z}^m +z$ where $k,m, n \in \mathbb{N}$ with $k,n>m$ and $b,c \in \mathbb{R}^+$ is a locally univalent and sense-preserving if and only if 
\begin{equation}
|z|\neq \left( \frac{1}{kb}\right)^\frac{1}{k-1}
\end{equation}
 and its dilatation function satisfies 
\begin{equation} 
| \omega (z)| =\frac{|nz^{n-1}+cmz^{m-1}|}{|bkz^{k-1}+1|} <1.
\end{equation} 
  Therefore a curve $\omega (z) =1$ separates the zeros in a sense-preserving region from the zeros in sense-reversing region and this region is determined  by $|nz^{n-1}+cmz^{m-1}|=|bkz^{k-1}+1|.$
\begin{theorem}$\label{III7}$
Let  $q(z) = bz^k+\overline{z}^n +c\overline{z}^m +z$ where $k,m, n \in \mathbb{N}$ such that $k,n>m=1$ and $b,c \in \mathbb{R}$ with $b \neq \pm 1.$ Suppose $n=k$ and $z^k\overline{z}$ be a pure imaginary number. Then $|\omega(z)|=1$ if and only if 
\begin{equation}
|z|= \left(\frac{1}{k}\right)^{\frac{1}{k-1}} \left(\frac{c^2-1}{b^2-1} \right)^{\frac{1}{2k-2}}
\end{equation}
 where $\omega(z)$ is a dilatation function of a quadrinomial $q(z).$ 
\end{theorem}
\begin{proof}
By definition, $|\omega(z)|=1$ iff $\frac{|\left(z^n+cz^m\right)'|}{|\left( bz^k+z \right)'|}=1.$ After a forward calculations we arrive at $|\omega(z)|=1$ iff 
% $$ \Leftrightarrow \left( nz^{n-1}+cmz^{m-1}\right)\left( n\overline{z}^{n-1}+cm\overline{z}^{m-1}\right)=\left( bkz^{k-1} +1 \right)\left(bk\overline{z}^{k-1}+1 \right). ~~~~~~~~~~~~~~~~~~~~~~~~~~~~~~~~~~~~~~~~~~~~~~~~~~~~~~~~~~~~~~~~~~~~~~~~~~~~~~~~~~~~~ $$ $$ \Leftrightarrow n^2|z|^{2(n-1)}+cmn\left[ z^{n-1}\overline{z}^{m-1} \right]+c^2m^2|z|^{2(m-1)}=b^2k^2|z|^{2(k-1)}+bk(z^{k-1}+\overline{z}^{k-})+1 .~~~~~~~~~~~~~~~~~~~~~~~~~~~~~~~~~~~~~~~~~~~~~~~~~~~~~~~~~~~~~~~~~~~~~~~~~~~ $$ $$ \Leftrightarrow n^2|z|^{2n}+cmn\left( z^n\overline{z}^m + \overline{z} ^n z^m \right)+c^2m^2|z|^{2m}=b^2k^2|z|^{2k}+bk(z^k\overline{z}+\overline{z}^kz)+|z|^2. ~~~~~~~~~~~~~~~~~~~~~~~~~~~~~~~~~~~~~~~~~~~~~~~~~~~~~~~~~~~~~~~~~~~~~~~~~~~~~~~~~~~$$ $$ \Leftrightarrow n^2|z|^{2n}+c^2m^2|z|^{2m}-b^2k^2|z|^{2k}-|z|^2 =  bk(z^k\overline{z}+\overline{z}^kz) - cmn\left( z^n\overline{z}^m \right). ~~~~~~~~~~~~~~~~~~~~~~~~~~~~~~~~~~~~~~~~~~~~~~~~~~~~~~~~~~~~~~~~$$ $$ \Leftrightarrow n^2|z|^{2n}+c^2m^2|z|^{2m}-b^2k^2|z|^{2k}-|z|^2 =  2bk Re(z^k\overline{z})-2cmn Re(z^n\overline{z}^m).  ~~~~~~~~~~~~~~~~~~~~~~~~~~~~~~~~~~~~~~~~~~~~~~~~~~~~~~~~~~~~~~~~ $$ 
$n^2|z|^{2n}+c^2m^2|z|^{2m}-b^2k^2|z|^{2k}-|z|^2 =  2bk \mathrm{Re}(z^k\overline{z})-2cmn \mathrm{Re}(z^n\overline{z}^m). $
 But here we have assumed that  $z^k\overline{z}$ is a pure imaginary number. As a result we do have the following. $$ |\omega (z)|=1 \Leftrightarrow n^2|z|^{2n}+c^2m^2|z|^{2m}-b^2k^2|z|^{2k}-|z|^2 = 0.~~~~~~~~~~~~~~~~~~~~~~~~~~~~~~~~~~~~~~~$$ $$ \Leftrightarrow k^2|z|^{2k-2}+c^2-b^2k^2|z|^{2k-2}=1. ~~~~~~~~~~~~~~~~~~~~~~~~~~~~~~~~~~~~~~~~~~~~~~~~~~~~~~~~~~~~~~~~$$ $$\Leftrightarrow |z|^{2k-2}= \left[\frac{1}{k^2} \left( \frac{c^2-1}{b^2-1} \right) \right]. ~~~~~~~~~~~~~~~~~~~~~~~~~~~~~~~~~~~~~~~~~~~~~~~~~~~~~~~~~~~~~~~~~~~~~~~~~~~~~~~~~~ $$ $$ \Leftrightarrow |z|=\left[ \frac{1}{k^2} \left(\frac{c^2-1}{b^2-1}\right)\right]^{\frac{1}{2k-2}}. ~~~~~~~~~~~~~~~~~~~~~~~~~~~~~~~~~~~~~~~~~~~~~~~~~~~~~~~~~~~~~~~~~~~~~~~~~~~~~~~~~~~~~~~~~~$$
\end{proof}
Now let us denote $\mathcal{M}_{b,c}:= \left(\frac{c^2-1}{k^2(b^2-1)}\right)^\frac{1}{2k-2}$ and let us introduce the following as  definition. 
\begin{definition}
The roots of the quadrinomial $q(z)=bz^k+\overline{z}^n +c\overline{z}^m +z$ that lie on the circle of radius $\mathcal{M}_{b,c}$ are said to be the $\mathcal{M}_{b,c}$\textit{-modular} roots.
\end{definition}

\section{Conclusion and Discussion}$\label{IV}$
In this paper  we have found the maximum number of the zeros of  of a two parameter family of harmonic quadrinomials $q(z) = bz^k+\overline{z}^n +c\overline{z}^m +z$ where $k,m, n \in \mathbb{N}$ with $k>n>m,$ and $ b,c \in \mathbb{R}$ to be $3n-2$  if ~~ $b=0$ and is $n^2-n+3k-2$ if $b\neq 0.$  Also we have derived a locations for the zeros and the result  shows that the location of the zeros also changes as a parameter in co-analytic part varies without any restriction to non-zero coefficients in analytic part. For this families of quadrinomials, also we have determined the curve $\Gamma _{b,c}= \lbrace z \in \mathbb{C}:|z|=\mathcal{M}_{b,c} =\left(\frac{c^2-1}{k^2(b^2-1)}\right)^\frac{1}{2k-2} \rbrace$ which separates zeros in sense-preserving region from the zeros in sense-reversing region by considering \bigskip the relation $k=n>m=1$.\\
For further investigation we have the following to be considered: 
\begin{itemize}
\item[(1)] How many roots of $q(z)$ are $\mathcal{M}_{b,c}$\textit{-modular}? What can be said on the number of zeros inside and outside of the circle $\Gamma _{b,c}= \{z \in \mathbb{C}:|z|=\mathcal{M}_{b,c} \}?$
\item[(2)] What is image of the circle $\Gamma _{b,c}= \{z \in \mathbb{C}:|z|=\mathcal{M}_{b,c} \}$ under the quadrinomial $q(z)=bz^k+\overline{z}^n +c\overline{z}^m +z?$
\end{itemize}
\section*{Acknowledgments}
The authors gratefully acknowledge Addis Ababa University for providing access to the completion of this work.
\section*{Declaration of Interest of Statement}
The authors declare that there are no conflicts of interest regarding the publication of this paper.
%\section*{References}

\end{document}